\documentclass[12pt,verbatim]{amsart}
\usepackage{amsfonts}
\usepackage{ifthen}
\usepackage{amsthm}
\usepackage{amsmath}
\usepackage{amssymb}
\usepackage{graphicx}
\usepackage{amscd,amssymb,amsthm}

\usepackage{caption}

\newcounter{minutes}
\divide\time by 60
\newcounter{hours}
\multiply\time by 60 \addtocounter{minutes}{-\time}

\title{Hyperbolic Lambert Quadrilaterals and Quasiconformal Mappings}

\author{Matti Vuorinen}
\author{Gendi Wang}

\address{Department of Mathematics and Statistics, University of Turku, Turku 20014,
Finland} \email{vuorinen@utu.fi}
\address{Department of Mathematics and Statistics, University of Turku, Turku 20014,
Finland} \email{genwan@utu.fi}

\newtheorem{theorem}[equation]{Theorem}

\newtheorem{lemma}[equation]{Lemma}
\newtheorem{proposition}[equation]{Proposition}

\newtheorem{corollary}[equation]{Corollary}
\newtheorem{remark}[equation]{Remark}

\newcommand{\beq}{\begin{equation}}
\newcommand{\eeq}{\end{equation}}

\newcommand{\BB}{\mathbb{B}^2}
\newcommand{\RR}{\mathbb{R}^2}
\newcommand{\UH}{\mathbb{H}^2}
\newcommand{\HH}{{\mathbb{H}}}
\newcommand{\R}{\mathbb{R}}
\newcommand{\Rn}{\mathbb{R}^n}
\newcommand{\arth}{{\rm arth}}
\newcommand{\SArc}{{\rm SArc}}
\newcommand{\Arc}{{\rm Arc}}

\numberwithin{equation}{section}

\begin{document}








\def\thefootnote{}
\footnotetext{ \texttt{\tiny File:~\jobname .tex,
          printed: \number\year-\number\month-\number\day,
          \thehours.\ifnum\theminutes<10{0}\fi\theminutes}
} \makeatletter\def\thefootnote{\@arabic\c@footnote}\makeatother

\maketitle

\begin{abstract}
We prove sharp bounds for the product and the sum of two hyperbolic distances
between the opposite sides of hyperbolic Lambert quadrilaterals in the unit disk.
Furthermore, we study the images of Lambert quadrilaterals under quasiconformal
mappings from the unit disk onto itself and obtain sharp results in this case, too.
\end{abstract}




\section{Introduction}

Given a pair of points in the closure of the unit disk ${\mathbb B}^2\,,$ there exists a unique hyperbolic geodesic line joining these two points.
Hyperbolic lines are simply sets of the form $C \cap \BB$, where $C$ is a circle perpendicular to the unit circle, or a Euclidean diameter of
$\BB\,.$ For a quadruple of four points
$\{a,b,c,d\}$ in the closure of the unit disk, we can draw these hyperbolic lines
joining each of the four pairs of points $\{a,b\},$ $\{b,c\},$ $\{c,d\},$ and $\{d,a\}\,.$ If these hyperbolic lines bound a domain $D \subset {\mathbb B}^2$  such that the points $\{a,b,c,d\}$ are in the positive order on the boundary of the domain, then we say
that the quadruple of points $\{a,b,c,d\}$ determines a hyperbolic quadrilateral $Q(a,b,c,d)\,$ and that the points $a,b,c,d$ are its vertices.
A hyperbolic quadrilateral with angles equal to $\pi/2, \pi/2, \pi/2, \phi\,(0\leq \phi < \pi/2)\,,$ is
called a hyperbolic {\em Lambert} quadrilateral \cite[p. 156]{be}, see Figure 1. Observe that one of the vertices of a Lambert quadrilateral may be on the unit circle, in which case the angle at that vertex is $\phi=0\,.$

In this paper, we study bounds for the product and the sum of
two hyperbolic distances between the opposite sides of hyperbolic Lambert quadrilaterals in the unit disk. Also, we consider the same product expression for the images of these hyperbolic Lambert quadrilaterals under quasiconformal mappings from the unit disk onto itself. In particular, we obtain similar results for ideal hyperbolic quadrilaterals, i.e., in the case when all the vertices are on the unit circle and all the angles are zero.
This follows, because an ideal hyperbolic quadrilateral can be subdivided
into four Lambert quadrilaterals.

For the formulation of our main results we introduce some notation --
further notation will be given below in Section 2.
Let $J^*[a,b]$ be the hyperbolic geodesic line with end points $a\,,b\in\partial\BB,$
and let $J[a,b]$ be the hyperbolic geodesic segment joining $a$ and $b$ when $a,b \in {\mathbb B}^2\,,$ or the hyperbolic geodesic ray when one of the two points $a\,,b$ is on $\partial\BB$.

Given two nonempty subsets $A, B$ of $ {\BB}$ (or of the upper
 half plane $ \UH\,$),
let $d_\rho(A,B)$ denote the hyperbolic distance between them,  defined as
 $$d_\rho(A,B)=\inf_{{x\in A\atop y\in B}}\rho(x,y)\,,$$
where $\rho(x,y)$ stands for the hyperbolic distance \eqref{th} (or \eqref{cosh} in the case $ A, B\subset\UH$).

We now formulate our main results.

\begin{theorem}\label{Lamdd}
Let $Q (v_a\,,v_b\,,v_c\,,v_d)$ be a hyperbolic Lambert quadrilateral in $\BB$ and let the quadruple of interior angles
$(\frac{\pi}{2}\,,\frac{\pi}{2}\,,\phi\,,\frac{\pi}{2})$, $\phi\in[0, \pi/2)\,,$ correspond to the quadruple $(v_a\,,v_b\,,v_c\,,v_d)$ of vertices. Let
$d_1=d_\rho(J[v_a,v_d]\,,J[v_b,v_c])\,,$  $d_2=d_\rho(J[v_a,v_b]\,,J[v_c,v_d])$ (see Figure \ref{Lamb}), and
let $L={\rm th} \rho(v_a,v_c)\in(0,1]$.
Then
\begin{eqnarray*}
d_1d_2\leq \left(\arth\left(\frac{\sqrt 2}{2}L\right)\right)^2 \,.
\end{eqnarray*}
The equality holds if and only if $v_c$ is on the bisector of the interior angle at $v_a$.
\end{theorem}

\medskip
\begin{figure}[h]
\centering
\includegraphics[width=7cm]{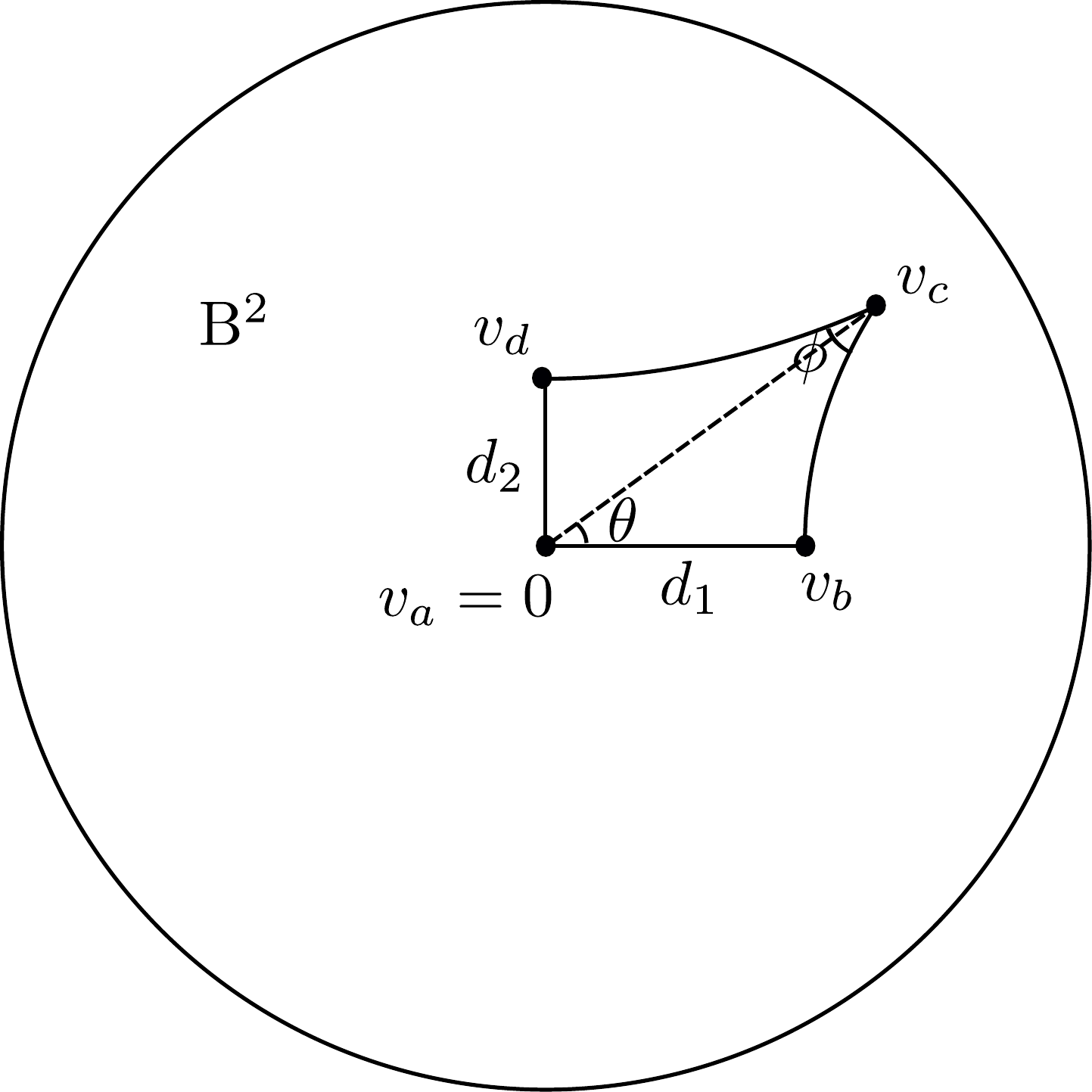}
\caption{\label{Lamb} A hyperbolic Lambert quadrilateral in $\BB$.}
\end{figure}
\medskip

\begin{theorem}\label{Lamdad}
Let $Q(v_a\,,v_b\,,v_c\,,v_d)$\,, $d_1$, $d_2$ and $L$ be as in Theorem \ref{Lamdd}.
Let $m= \sqrt{(2-L^2)(3L^2-2)}$, $r_0=\sqrt{\frac{1-m/L^2}{2}}$ and $r'_0=\sqrt{1-r^2_0}$.

(1) If $0<L\leq\sqrt{\frac 23}$, then
\begin{eqnarray*}
\arth\,L<d_1+d_2\leq \arth\left(\frac{2\sqrt 2L}{2+L^2}\right).
\end{eqnarray*}
The equality holds in the right-hand side if and only if $v_c$ is on the bisector of the interior angle at $v_a$.

(2) If $\sqrt{\frac 23}<L<\sqrt{2(\sqrt 2-1)}$, then
\begin{eqnarray*}
\arth\,L<d_1+d_2\leq \arth\left(\frac{L(r_0+r'_0)}{1+L^2r_0r'_0}\right).
\end{eqnarray*}
The equality holds in the right-hand side if and only if the interior angle between $J[v_a,v_b]$ and $J[v_a,v_c]$ is $\arccos r_0$ or $\arccos r'_0$.

(3) If $\sqrt{2(\sqrt 2-1)}\leq L<1$, then
\begin{eqnarray*}
\arth\left(\frac{2\sqrt 2L}{2+L^2}\right)\leq d_1+d_2\leq \arth\left(\frac{L(r_0+r'_0)}{1+L^2r_0r'_0}\right).
\end{eqnarray*}
The equality holds in the left-hand side if and only if $v_c$ is on the bisector of the interior angle at $v_a$.
The equality holds in the right-hand side if and only if the interior angle between $J[v_a,v_b]$ and $J[v_a,v_c]$ is $\arccos r_0$ or $\arccos r'_0$.

(4) If $L=1$, then
\begin{eqnarray*}
d_1+d_2\geq \arth\left(\frac{2\sqrt 2}{3}\right)\,.
\end{eqnarray*}
The equality holds if and only if $v_c$ is on the bisector of the interior angle at $v_a$.
\end{theorem}

In a Lambert quadrilateral, the angle $\phi$ is related to the lengths
$d_1$, $d_2$ of the sides "opposite" to it as follows
\cite[Theorem 7.17.1]{be}:
$${\rm sh}\, d_1{\rm sh}\, d_2=\cos\phi.$$
See also the recent paper of A. F. Beardon and D. Minda \cite[Lemma 5]{bm}.
The proof of Theorem \ref{Lamdd} yields the following corollary, which provides a connection between $d_1$, $d_2$ and $L={\rm th} \rho(v_a,v_c)$.

\begin{corollary} Let $L$, $d_1$ and $d_2$ be as in Theorem \ref{Lamdd}. Then
\begin{eqnarray*}
{\rm th}^2\,d_1+{\rm th}^2\,d_2=L^2.
\end{eqnarray*}
\end{corollary}

By Theorem \ref{Lamdd} and Theorem \ref{Lamdad}, we obtain the following corollary which deals with the ideal hyperbolic quadrilaterals.

\begin{corollary}\label{dd}
Let $Q(a,b,c,d)$ be an ideal hyperbolic quadrilateral in $\BB$. Let $d_1=d_\rho(J^*[a,d],J^*[b,c])$ and
$d_2=d_\rho(J^*[a,b], J^*[c,d])$ (see Figure \ref{thdd}).
Then
\begin{eqnarray*}
d_1d_2\leq \left(2\log(\sqrt{2}+1)\right)^2
\end{eqnarray*}
and
\begin{eqnarray*}
d_1+d_2\geq 4\log(\sqrt{2}+1)\,.
\end{eqnarray*}
In both cases the equalities hold if and only if $|a,b,c,d|=2$.
\end{corollary}

\medskip
\begin{figure}[h]
\centering
\includegraphics[width=7cm]{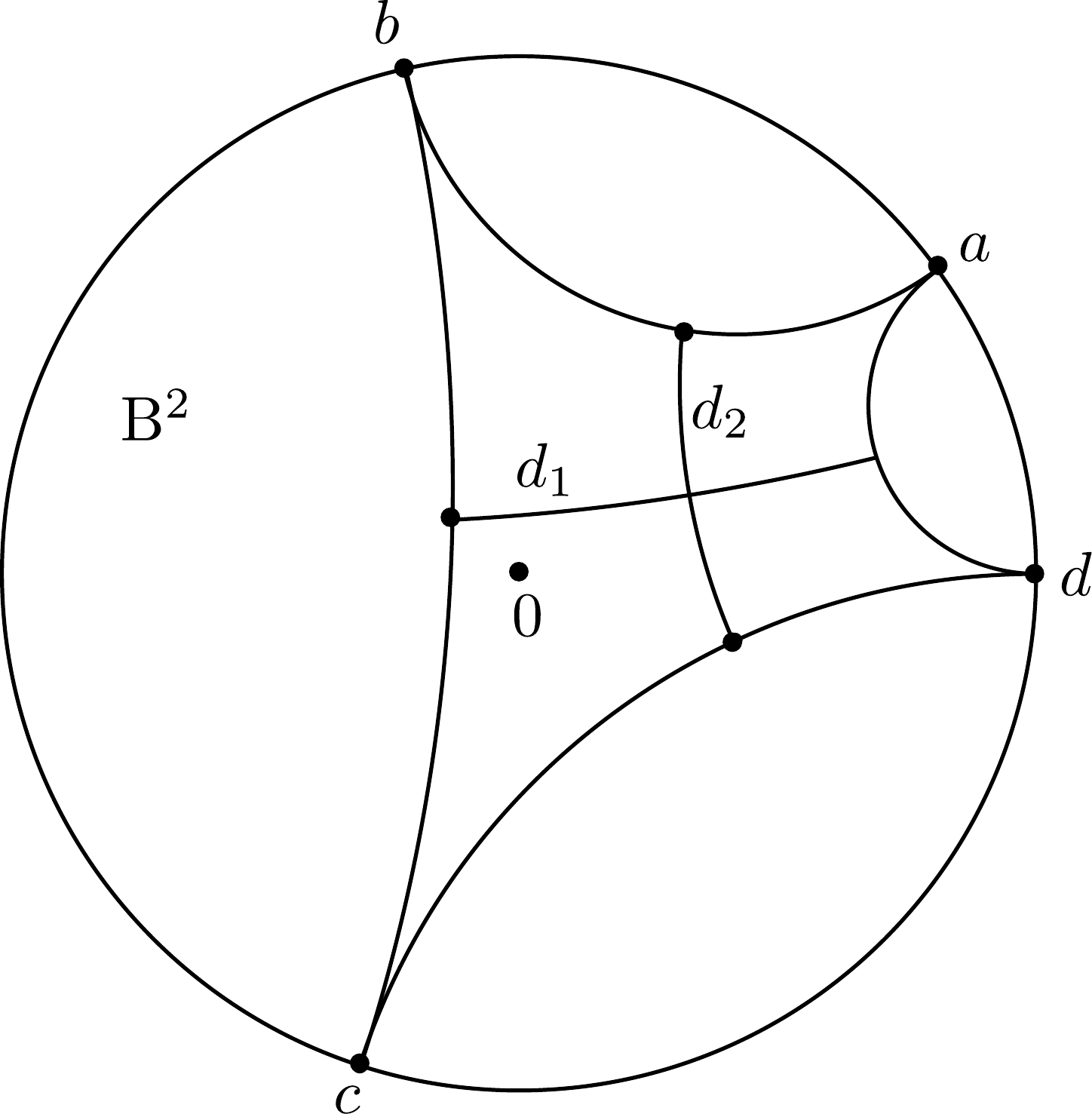}
\caption{\label{thdd} An ideal hyperbolic quadrilateral in $\BB$.}
\end{figure}
\medskip

\begin{remark}
$|a,b,c,d|=2$ means that there exists a M\"obius transformation $f$ such that $f(a)=1$,$f(b)=i$,$f(c)=-1$,$f(d)=-i$, see \eqref{crossratio}.
\end{remark}

\begin{theorem}\label{thqcLam}
Let $f: \BB\rightarrow\BB$ be a $K$-quasiconformal mapping with $f \BB= \BB$
and let $Q(v_a, v_b, v_c, v_d)$, $d_1\,, d_2$, $L$ be as in Theorem \ref{Lamdd}. Let $A(K)$ be as in Lemma \ref{leqc} and $f_L(r)$ be as in Lemma \ref{lecr}(1) by taking $c=L$.
Denote $D_1=d_\rho(f(J[v_a,v_d]),f(J[v_b,v_c]))$ and
$D_2=d_\rho(f(J[v_a,v_b]), f(J[v_c,v_d]))\,.$

(1) If $0<L\leq \frac{e^2-1}{e^2+1}\approx0.761594$, then
$$D_1D_2\leq A(K)^2 \left(\arth\left(\frac{\sqrt 2}{2}L\right)\right)^{2/K} .$$

(2) If $\frac{e^2-1}{e^2+1}<L\leq1$, then let $r_L=\frac 1L\frac{e^2-1}{e^2+1}\approx\frac{0.761594}{L}$ and $$M_L=\frac{f_L(\sqrt{1-r^2_L})}{f_L(r_L)}>1.$$
Let $r_L(K)$ be the unique solution $r$ to the equation $K f_L(r)=f_L(\sqrt{1-r^2})$ with $r_L<r<1$.
Further, define
$$T(x,L)=\arth(L x)\left(\arth\left(L\sqrt{1- x^2}\right)\right)^{1/K}\,,\,0<x<1.$$
Then
\begin{eqnarray*}
D_1D_2\leq A(K)^2\max\left\{T(r_L(K),L),\,\, \left(\arth\left(\frac{\sqrt 2}{2}L\right)\right)^{2/K} \right\}
\end{eqnarray*}
if $K> M_L$, and
\begin{eqnarray*}
D_1D_2\leq A(K)^2\max\left\{T(r_L,L),\,\, \left(\arth\left(\frac{\sqrt 2}{2}L\right)\right)^{2/K} \right\}
\end{eqnarray*}
if $1\leq K\leq M_L$.
\end{theorem}

\begin{corollary}\label{thqc}
Let $f: \BB\rightarrow\BB$ be a $K$-quasiconformal mapping with $f \BB= \BB\,$
and let $Q(a,b,c,d)$, $d_1, d_2$ be as in Corollary \ref{dd}. Let $A(K)$ be as in Lemma \ref{leqc} and $f_1(r)$ be as in Lemma \ref{lecr}(1)
by taking $c=1$. Denote $D_1=d_\rho(f(J^*[a,d]), f(J^*[b,c]))$ and
$D_2=d_\rho(f(J^*[a,b]),f(J^*[c,d])).$
 Further denote $r_1=\frac{2 \sqrt e}{e+1}\approx0.886819$ and
$$M_1=\frac{(e-1)(\log(\sqrt e+1)-\log(\sqrt e-1))}{ \sqrt e}\approx 1.46618\,$$
and
define $r_1(K)$ to be the unique solution $r$ to the equation $K f_1(r)=f_1(\sqrt{1-r^2})$ with $r_1<r<1$.
With the notation
$$T(x)=\arth(x)\left(\arth(\sqrt{1- x^2})\right)^{1/K}\,,\,0<x<1\,,$$
we have
\begin{eqnarray*}
D_1D_2\leq A(K)^2\max\left\{2^{1+1/K}T(r_1(K)),\,\, \left(2\log(\sqrt{2}+1)\right)^2 \right\}
\end{eqnarray*}
if $K> M_1$, and
\begin{eqnarray*}
D_1D_2\leq A(K)^2\max\left\{2^{1+1/K}T(r_1),\,\, \left(2\log(\sqrt{2}+1)\right)^2\right\}
\end{eqnarray*}
if $1\leq K\leq M_1$.
\end{corollary}

\section{Preliminaries}

It is assumed that the reader is familiar with basic definitions of geometric
function theory and quasiconformal mapping theory, see e.g. \cite{be,v}.
We recall here some basic information on hyperbolic geometry \cite{be}.

The chordal metric is defined by
\beq\label{q}
\left\{\begin{array}{ll}
q(x,y)=\frac{|x-y|}{\sqrt{1+|x|^2}\sqrt{1+|y|^2}},&\,\,\, x\,,y\neq\infty,\\
q(x,\infty)=\frac{1}{\sqrt{1+|x|^2}},&\,\,\, x\neq\infty,
\end{array}\right.
\eeq
for $x,y\in\overline{\RR}$.

For an ordered quadruple $a,b,c,d$ of distinct points in $\overline{\RR}$ we define the absolute ratio by
$$|a,b,c,d|=\frac{q(a,c)q(b,d)}{q(a,b)q(c,d)}.$$
It follows from (\ref{q}) that for distinct points $a,b,c,d\in \RR$
\beq\label{crossratio}
|a,b,c,d|=\frac{|a-c||b-d|}{|a-b||c-d|}.
\eeq
The most important property of the absolute ratio is M\"obius invariance, see \cite[Theorem 3.2.7]{be}, i.e., if $f$ is a M\"obius transformation, then
$$|f(a),f(b),f(c),f(d)|=|a,b,c,d|,$$
for all distinct $a,b,c,d\in\overline{\RR}$.

For a domain $G\subsetneq \RR$ and a continuous weight function $w: G\rightarrow(0,\infty)\,,$  we define the weighted length of a rectifiable curve $\gamma\subset G$ to be
$$\ell_w(\gamma)=\int_{\gamma}w(z)|dz|$$
and the weighted distance between two points $x,y \in G $ by
$$d_w(x,y)=\inf_{\gamma}\ell_w(\gamma),$$
where the infimum is taken over all rectifiable curves in $G$ joining $x$ and $y$ ($x=(x_1,x_2),\,y=(y_1,y_2)$). It is easy to see that $d_w$ defines a metric on $G$ and $(G,d_w)$ is a metric space. We say that a curve $\gamma: [0,1]\rightarrow G$ is a geodesic joining $\gamma(0)$ and $\gamma(1)$ if for all $t\in (0,1)$, we have
$$d_w(\gamma(0),\gamma(1))=d_w(\gamma(0),\gamma(t))+d_w(\gamma(t),\gamma(1)).$$
The hyperbolic distance in $\UH$ and $\BB$ is defined in terms of the weight functions
$w_{\UH}(x)=1/{x_2}$ and  $w_{\BB}(x)=2/{(1-|x|^2)}\,,$ resp.  We also have the corresponding explicit formulas
\beq\label{cosh}
\cosh\rho_{\UH}(x,y)=1+\frac{|x-y|^2}{2x_2y_2}
\eeq
for all $x,y\in \UH$  \cite[p.35]{be}, and
\beq\label{th}
{\rm th}\frac{\rho_{\BB}(x,y)}{2}=\frac{|x-y|}{\sqrt{|x-y|^2+(1-|x|^2)(1-|y|^2)}}
\eeq
for all $x,y\in \BB$  \cite[p.40]{be}. In particular, for $t\in(0,1)$,
\beq\label{arth}
\rho_{\BB}(0,t e_1)=\log\frac{1+t}{1-t}=2\arth t.
\eeq

There is a third equivalent way to express the hyperbolic distances.
Let $G\in\{\UH,\BB\}$, $x,y\in{G}$ and let $L$ be an arc of a circle
perpendicular to $\partial G$ with $x,y\in L$ and let
$\{x_*,y_*\}=L\cap\partial G$, the points being labelled so that $x_*, x, y, y_*$ occur in this order on $L$. Then by \cite[(7.2.6)]{be}
\beq\label{rho}
\rho_G(x,y)=\sup\{\log|a,x,y,b|:a,b\in\partial G\}=\log|x_*,x,y,y_*|.
\eeq
We will omit the subscript $G$ if it is clear from the context.
The hyperbolic distance is invariant under M\"obius transformations of $G$ onto $G'$ for $G,\,G'\in\{\UH,\BB\}$.

Hyperbolic geodesics are  arcs of circles which are orthogonal to the boundary of the domain. More precisely, for $a,b\in \BB$ (or $\UH)$, the hyperbolic geodesic segment joining $a$ to $b$ is an arc of a circle orthogonal to $S^1$ (or $\partial \UH)$. In a limiting case the points $a$ and $b$ are located on a Euclidean line through $0$ (or located on a normal of $\partial \UH$), see \cite{be}.  Therefore, the points $x_*$ and $y_*$ are the end points of the hyperbolic geodesic. For any two distinct points the hyperbolic geodesic segment is unique (see Figure \ref{h2} and \ref{b2}). For basic facts about hyperbolic geometry we refer the interested reader to \cite{a}, \cite{be} and \cite{kl}.

\medskip
\begin{figure}[h]
\begin{minipage}[t]{0.45\linewidth}
\centering
\includegraphics[width=8cm]{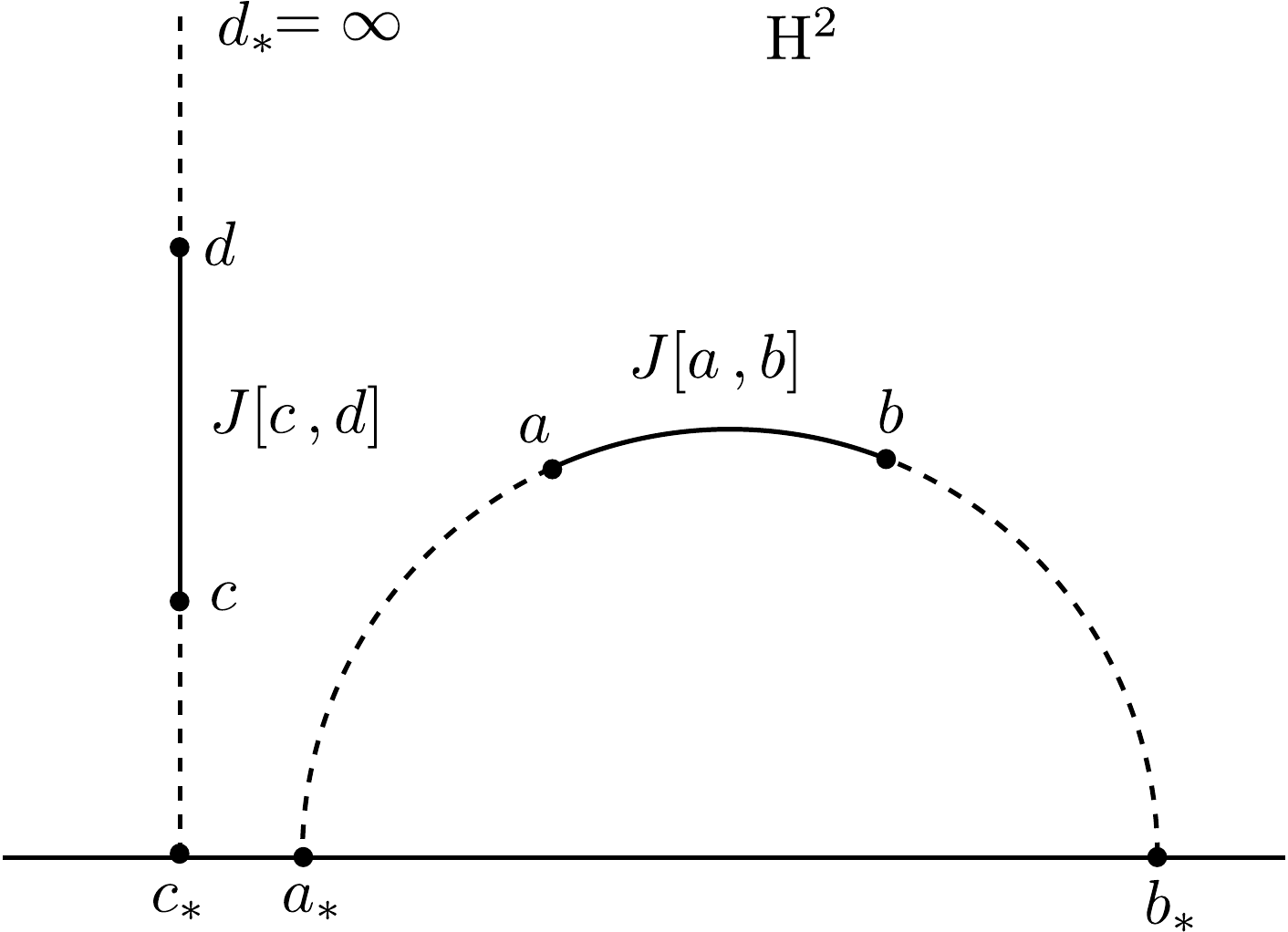}
\caption{\label{h2} Hyperbolic geodesic segments in $\UH$.}
\end{minipage}
\hfill
\hspace{1cm}
\begin{minipage}[t]{0.45\linewidth}
\centering
\includegraphics[width=6cm]{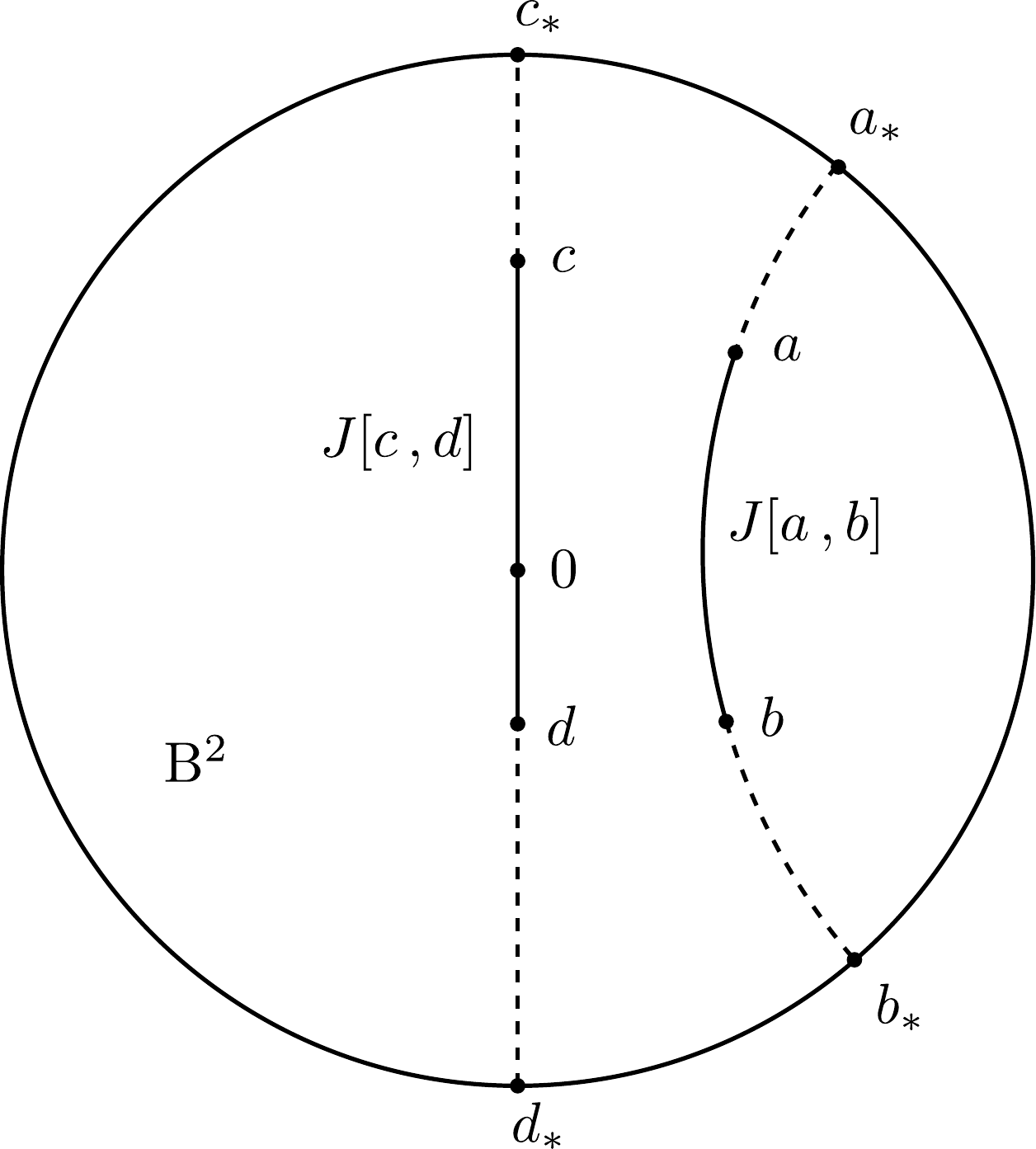}
\caption{\label{b2} Hyperbolic geodesic segments in $\BB$.}
\end{minipage}
\end{figure}
\medskip

By \cite[Exercise 1.1.27]{k} and \cite[Lemma 2.2]{kv}, for  $x\,,y \,\in \mathbb{R}^2\setminus\{0\}$ such that $0,x,y$ are noncollinear,
the circle $S^1(a,r_a)$ containing $x, y$ is orthogonal to the unit circle, where
\beq\label{orar}
a =i\frac{y(1+|x|^2)-x(1+|y|^2)}{2(x_2y_1-x_1y_2)}\,\,\,\, {\rm and}\, \,\,\,r_a=\frac {|x-y|\big|x|y|^2-y\big|}{2|y||x_1y_2-x_2y_1|}\,.
\eeq

For $r,s\in(0,+\infty)$, the \emph{H\"older mean of order $p$} is defined by
$$H_p(r,s)=\left(\frac{r^p+s^p}{2}\right)^{1/p} \quad\mbox{for}\quad{p\neq 0},
\quad H_0(r,s)=\sqrt{r\,s}.$$
For $p=1$, we get the arithmetic mean $A=H_1$; for $p=0$, the
geometric mean $G=H_0$; and for $p=-1$, the harmonic mean $H=H_{-1}$.
It is well-known that $H_p(r,s)$ is continuous and increasing with respect to $p$.
Many interesting properties of H\"older means are given in \cite{bu} and \cite{hlp}.

A function $f:I\to J$ is called {\it $H_{p,q}$-convex (concave)} if it satisfies
$$f(H_p(r,s))\leq(\geq)H_q(f(r),f(s))$$
for all $r,s\in I$, and  {\it strictly $H_{p,q}$-convex (concave)} if the
inequality is strict except for $r=s$. For $H_{p,q}$-convexity
of some special functions the reader is referred to \cite{avv2, ba, wzj1, wzj2, zwc}.

Some other notation is also needed in the paper. Let $[a,b]$ be the Euclidean segment with end points $a$ and $b$.
Let $X, Y$ be the real axis and imaginary axis, resp.
Let $\Arc(abc)$ be the circular arc with end points $a,c$ and through $b$, and $\SArc(ac)$ be semicircle with end point $a,c$.

\medskip

The next lemma, so-called {\em
monotone form of l'H${\rm \hat{o}}$pital's rule}, has found recently numerous applications in proving inequalities. See the extensive bibliography of \cite{avz}.

\begin{lemma} \label{lhr}{\rm \cite[Theorem 1.25]{avv1}}
For $-\infty<a<b<\infty$, let $f,\,g: [a,b]\rightarrow \mathbb{R}$ be continuous on $[a,b]$, and be differentiable on $(a,b)$, and let $g'(x)\neq 0$ on $(a,b)$. If $f'(x)/g'(x)$ is increasing (deceasing) on $(a,b)$, then so are
\begin{eqnarray*}
\frac{f(x)-f(a)}{g(x)-g(a)}\,\,\,\,\,\,\,and\,\,\,\,\,\,\,\,\frac{f(x)-f(b)}{g(x)-g(b)}.
\end{eqnarray*}
If $f'(x)/g'(x)$ is strictly monotone, then the monotonicity in the conclusion is also strict.
\end{lemma}

From now on we let $r'=\sqrt{1-r^2}$ for $0<r<1$.

\begin{lemma}\label{lecr}
Let $c\in(0,1]$ and $r\in(0,1)$.

(1) The function $f_c(r)\equiv\frac{1-(c r')^2}{r\arth(c r)}$ is strictly decreasing and concave with range $(0,1)$ if $c=1$, and strictly decreasing with range $(0,\infty)$ if $0<c<1$.

(2) The function $F_c(r)\equiv\arth(c r) \arth(c r')$ is strictly increasing on $(0,\frac{\sqrt{2}}{2}]$ and strictly decreasing on $[\frac{\sqrt{2}}{2}, 1)$ with maximum value $(\arth(\frac{\sqrt{2}}{2} c))^2$.
\end{lemma}

\begin{proof}

(1) If $c=1$, then $f_1(r)=\frac{r}{\arth\,r}$. By differentiation,
$$f'_1(r)=\frac{h_{11}(r)}{h_{12}(r)},$$
where $h_{11}(r)=\arth\, r-\frac{r}{r'^2}$ and $h_{12}(r)=(\arth\, r)^2$. It is easy to see that $h_{11}(0^+)=h_{12}(0^+)=0$. Then
\begin{eqnarray*}
\frac{h'_{11}(r)}{h'_{12}(r)}=-\frac{h_{13}(r)}{h_{14}(r)},
\end{eqnarray*}
where $h_{13}(r)=(\frac{r}{r'})^2$ and $h_{14}(r)=\arth\, r$. Then $h_{13}(0^+)=h_{14}(0^+)=0$. By differentiation, we have
\begin{eqnarray*}
\frac{h'_{13}(r)}{h'_{14}(r)}=\frac{2r}{r'^2},
\end{eqnarray*}
which is strictly increasing. Hence by Lemma \ref{lhr}, $\frac{h'_{11}(r)}{h'_{12}(r)}$ is strictly decreasing and so is $f'_1$
with $f'_1(r)<f'_1(0^+)=0$. Therefore, $f_1$ is strictly decreasing and concave on $(0,1)$. The limiting value $f_1(1^-)=0$ is clear and $f_1(0^+)=1$ by l'H${\rm \hat{o}}$pital's Rule.

If $0<c<1$, then
$$f_c(r)=f_1(cr)h(r),$$
where $h(r)=\frac{1-(cr')^2}{cr^2}$. By differentiation,
\begin{eqnarray*}
h'(r)=-\frac{2(1-c^2)}{cr^3}<0.
\end{eqnarray*}
Therefore, $f_c$ is strictly decreasing. The limiting values are clear.
\medskip

(2) By differentiation,
$$F'_c(r)=\frac{c}{r'f_c(r)}\left(\frac{f_c(r)}{f_c(r')}-1\right).$$
$\frac{f_c(r)}{f_c(r')}$ is strictly decreasing from $(0,1)$ onto $(0,\infty)$ by (1) and $F'_c(\frac{\sqrt{2}}{2})=0$. Then
$F_c$ is strictly increasing on $(0,\frac{\sqrt{2}}{2}]$ and strictly decreasing on $[\frac{\sqrt{2}}{2}, 1)$ with maximum value $F_c(\frac{\sqrt{2}}{2})$.
\end{proof}

\begin{lemma}\label{lecra}
Let $c\in(0,1]$, $r\in(0,1)$, $m=\sqrt{(2-c^2)(3c^2-2)}$ and $r_0=\sqrt{\frac{1-m/c^2}{2}}$. Let
$$G_c(r)\equiv\arth(c r)+\arth(c r').$$

(1) If $0<c\leq\sqrt{\frac 23}$, then the range of $G_c$ is $(\arth\, c, \arth(\frac{2\sqrt 2 c}{2+c^2})]$.

(2) If $\sqrt{\frac 23}<c<\sqrt{2(\sqrt 2-1)}$, then the range of $G_c$ is $(\arth\, c, \arth (\frac{c (r_0+r'_0)}{1+c^2r_0r'_0})]$.

(3) If $\sqrt{2(\sqrt 2-1)}\leq c<1$, then the range of $G_c$ is $[\arth (\frac{2\sqrt 2 c}{2+c^2}), \arth (\frac{c (r_0+r'_0)}{1+c^2r_0r'_0})]$.

(4) If $c=1$, then the range of $G_c$ is $[\arth(\frac{2\sqrt{2}}{3}), \infty)$.

\end{lemma}

\begin{proof}
It is clear that the limiting values
$$
G_c(0^+)=G_c(1^-)=
\left\{\begin{array}{ll}
\arth\, c,&\,\,\,0<c<1,\\
\infty,&\,\,\,c=1.
\end{array}\right.
$$
Let
$$g_c(r)={\rm th}G_c(r)=\frac{c(r+r')}{1+c^2 r r'}\,.$$
By differentiation, we have
$$
r'(1+c^2r r')^2 g'_c(r)=c(r'-r)(1-c^2-c^2 r r').
$$
Making substitution of $x=r^2$, we get
\begin{eqnarray}\label{cx}
& &1-c^2-c^2 r r'=0\\
&\Leftrightarrow&
c^4x^2-c^4x+(1-c^2)^2=0\,.\nonumber
\end{eqnarray}
Therefore, equation (\ref{cx}) has no root if $0<c^2< \frac 23$ or $c=1$ ,
only one root $\frac{\sqrt{2}}{2}$ if $c^2=\frac 23$,
and two different roots $r_0\,,r'_0$ if $\frac 23<c^2<1$. It is obvious that $r_0\in (0,\frac{\sqrt 2}{2})$.

It is easy to see that $c=g_c(0)<g_c(\frac{\sqrt{2}}{2})=\frac{2\sqrt 2}{2+c^2}c$ if and only if $c^2<2(\sqrt 2-1)\approx0.828427$. If $\sqrt{\frac 23}<c<1$, $g_c$ is increasing on $(0,r_0)$ , $(\frac{\sqrt{2}}{2}, r'_0)$ and decreasing on $(r_0, \frac{\sqrt{2}}{2})$, $(r'_0,1)$.

Now we get the following conclusions.

(1) If $0<c\leq\sqrt{\frac 23}$, $g'_c(r)=0$ is equivalent to $r=\frac{\sqrt{2}}{2}$.
Then
$$\arth\, c=\arth (g_c(0))<G_c(r)\leq \arth (g_c({\sqrt{2}}/{2}))=\arth\left(\frac{2\sqrt 2 c}{2+c^2}\right).$$

(2) If $\sqrt{\frac 23}<c<\sqrt{2(\sqrt 2-1)}$, then
 $$\arth\, c=\arth (g_c(0))<G_c(r)\leq\arth (g_c(r_0))=\arth\left(\frac{c (r_0+r'_0)}{1+c^2r_0r'_0}\right).$$

(3) If $\sqrt{2(\sqrt 2-1)}\leq c<1$, then
$$\arth \left(\frac{2\sqrt 2 c}{2+c^2}\right)=\arth (g_c({\sqrt{2}}/{2}))\leq G_c(r)\leq\arth (g_c(r_0))=\arth\left(\frac{c (r_0+r'_0)}{1+c^2r_0r'_0}\right).$$

(4) If $c=1$, then
$$G_c(r)\geq \arth(g_1({\sqrt{2}}/{2}))=\arth({2\sqrt{2}}/{3}).$$

This completes the proof.
\end{proof}

\begin{lemma}\label{le1} Let $r\in(0,1)$.

(1) The function $h_1(r)\equiv \frac{r'}{\arth\,r'}$ is strictly increasing and concave with range $(0,1)$.

(2) The function $h(r)\equiv \frac{r}{\arth\,r}+\frac{r'}{\arth\,r'}$ is strictly increasing on $(0,\frac{\sqrt 2}{2}]$, strictly decreasing on  $[\frac{\sqrt 2}{2},1)$, and concave on $(0,1)$ with range $(1, \frac{\sqrt2}{\log(\sqrt2+1)}]$.
\end{lemma}

\begin{proof}

(1) The monotonicity and the limiting values of $h_1$ can be easily obtained  by Lemma \ref{lecr}(1).

Now we prove the concavity of $h_1$. By differentiation,
$$h'_1(r)=\frac{r'-r^2\arth\,r'}{rr'(\arth\,r')^2}$$
and
\begin{eqnarray}\label{le13}
h''_1(r)r^2r'^3(\arth\,r')^3=2r'^2-r'\arth\,r'-r^2(\arth\,r')^2\equiv\psi_1(r').
\end{eqnarray}
Then $\psi_1(r)=2r^2-r\arth\,r-r'^2(\arth\,r)^2$ and by differentiation
\begin{eqnarray*}
\psi'_1(r)=r\left(4-\psi_2(r)\right),
\end{eqnarray*}
where $\psi_2(r)=\frac{1}{r'^2}+3\frac{\arth\,r}{r}-2(\arth\,r)^2$.
Since
\begin{eqnarray*}
\arth\,r=\frac12\log\frac{1+r}{1-r}=\sum_{n=0}^{\infty}\frac{r^{2n+1}}{2n+1},
\end{eqnarray*}
we have
\begin{eqnarray*}
r^2r'^4\psi'_2(r)
&=&(r^4+2r^2-3)\arth\,r-r^3+3r\nonumber\\
&=&\sum_{n=0}^{\infty}\frac{r^{2n+5}}{2n+1}+2\sum_{n=1}^{\infty}\frac{r^{2n+3}}{2n+1}-3\sum_{n=2}^{\infty}\frac{r^{2n+1}}{2n+1}\\
&=&\sum_{n=2}^{\infty}\frac{16(n-1)}{(2n-3)(2n-1)(2n+1)}r^{2n+1}>0.
\end{eqnarray*}
Hence $\psi_2$ is increasing with $\psi_2(0^+)=4$ and  $\psi_1$ is decreasing with $\psi_1(0^+)=0$. Therefore by (\ref{le13}), $h''_1$ is negative  and $h'_1$ is decreasing. Then $h_1$ is concave on $(0,1)$.

\medskip
(2) By differentiation,
$$h'(r)=f'_1(r)+h'_1(r)=f'_1(r)-\frac{r}{r'}f_1'(r'),$$
where $f_1(r)=\frac{r}{\arth\,r}=h_1(r')$.
It is easy to see that $h'(\frac{\sqrt2}{2})=0$. Therefore, $h$ is strictly increasing on $(0,\frac{\sqrt 2}{2}]$ and strictly decreasing on  $[\frac{\sqrt 2}{2},1)$. By (1) and Lemma \ref{lecr}(1), $h'$ is strictly decreasing and $h(0^+)=h(1^-)=1$. Hence $h$ is concave and
$$1<h(r)\leq h({\sqrt2}/{2})=\frac{\sqrt2}{\log(\sqrt2+1)}.$$

This completes the proof.
\end{proof}

\begin{lemma}\label{le2}
Let $p\in\R$, $r\in(0,1)$, and $C=1-\frac{\log(\sqrt{2}+1)}{\sqrt2}\approx 0.376775 $. Let
$$g(r)=\frac{r}{r'}\left(\frac{\arth\,r}{\arth\,r'}\right)^{p-1}.$$

(1) $g$ is strictly decreasing if $p\leq 0$ and strictly increasing if $p\geq C$.

(2) If $p\in(0,C)$, then there exists exactly one point $r_0\in(0,\frac{\sqrt2}{2})$ such that $g$ is increasing on $(0,r_0)$, $(r'_0,1)$ and decreasing on $(r_0,r'_0)$.

(3) If $p\in(0,C)$, then $g(0^+)=0$ and $g(1^-)=\infty$.
\end{lemma}

\begin{proof}
(1) By logarithmic differentiation,
\begin{eqnarray*}
\frac{rr'^2\arth\,r\arth\,r'}{r\arth\,r'+r'\arth\,r}\cdot\frac{g'(r)}{g(r)}&=&p-1+\frac{\arth\,r\arth\,r'}{r\arth\,r'+r'\arth\,r}\\
&=&p-\left(1-\frac{1}{h(r)}\right),
\end{eqnarray*}
where $h(r)$ is as in Lemma \ref{le1}(2). Since
$$0<1-\frac{1}{h(r)}\leq C,$$
we see that
$g$ is strictly increasing if $p\geq C$ and decreasing if $p\leq 0$.

\medskip
(2) If $p\in(0,C)$, then there exists exactly one point $r_0\in(0,\frac{\sqrt2}{2})$ such that $g'(r_0)=g'(r'_0)=0$ because $1-\frac{1}{h(r)}$ is increasing on $(0,\frac{\sqrt2}{2})$ and decreasing on $(\frac{\sqrt2}{2},1)$ by Lemma \ref{le1}(2). Therefore, $g'>0$ if $r\in(0,r_0)\cup(r'_0,1)$ and $g'<0$ if $r\in(r_0,r'_0)$.  Hence $g$ is increasing on $(0,r_0)$, $(r'_0,1)$ and decreasing on $(r_0,r'_0)$.

\medskip
(3) Since $0<p<C<1$ and
\begin{eqnarray*}
\lim\limits_{r\rightarrow 0^+}\frac{(\arth\,r')^{1-p}}{r^{-p}}
=\lim\limits_{r\rightarrow 0^+}\frac{1-p}{p}\cdot\frac{r^p}{r'(\arth\,r')^p}
=0,
\end{eqnarray*}
we have
\begin{eqnarray*}
\lim\limits_{r\rightarrow 0^+}g(r)=\lim\limits_{r\rightarrow 0^+}\left(\frac{r}{\arth\,r}\right)^{1-p}\cdot\left(\frac{1}{r'}\right)\cdot\frac{(\arth\,r')^{1-p}}{r^{-p}}=0
\end{eqnarray*}
and
\begin{eqnarray*}
\lim\limits_{r\rightarrow 1^-}g(r)=\lim\limits_{r\rightarrow 1^-}\left(\frac{\arth\,r'}{r'}\right)^{1-p}\cdot r\cdot\frac{r'^{-p}}{(\arth\,r)^{1-p}}=\infty.
\end{eqnarray*}
This completes the proof.
\end{proof}

\begin{theorem}\label{th1}
Let $C$ be as in Lemma \ref{le2}. Then for all $r\in(0,1)$,
\begin{eqnarray}\label{th11}
H_p(\arth\,r,\arth\,r')\leq \arth\left(\frac{\sqrt2}{2}\right)
\end{eqnarray}
holds if and only if $p\leq 0$, and
\begin{eqnarray}\label{th12}
H_p(\arth\,r,\arth\,r')\geq \arth\left(\frac{\sqrt2}{2}\right)
\end{eqnarray}
holds if and only if $p\geq C$.
The equalities hold if and only if $r=r'=\frac{\sqrt2}{2}$ and all inequalities are sharp in both cases.
\end{theorem}
\begin{proof}
We immediately obtain the inequality (\ref{th11}) if $p=0$ by Lemma \ref{lecr}(2). Therefore, it suffices to discuss the case $p\neq 0$.

Let
$$f(r)=\frac 1p\log\frac{(\arth\,r)^p+(\arth\,r')^p}{2}\,,\,\,p\neq 0.$$
By differentiation,
\begin{eqnarray*}
f'(r)=\frac{(\arth\,r')^{p-1}}{rr'\left((\arth\,r)^p+(\arth\,r')^p\right)}\left(g(r)-1\right),
\end{eqnarray*}
where $g(r)$ is as in Lemma \ref{le2} and it is easy to see that $g(\frac{\sqrt2}{2})=1$.

{\it Case 1.} $p<0$.
$f$ is strictly increasing on $(0,\frac{\sqrt 2}{2})$ and strictly decreasing on $(\frac{\sqrt 2}{2},1)$ by Lemma \ref{le2}(1).
Therefore,  $f(r)\leq f(\frac{\sqrt2}{2})$ for all $p<0$.

{\it Case 2.} $p\geq C$.
By Lemma \ref{le2}(1), $f$ is strictly decreasing on $(0,\frac{\sqrt 2}{2})$ and strictly increasing on $(\frac{\sqrt 2}{2},1)$, and hence $f(r)\geq f(\frac{\sqrt2}{2})$.

{\it Case 3.} $p\in(0,C)$.
By Lemma \ref{le2}(2), there exists exactly one point $r_1\in(0,r_0)$ such that $g(r_1)=g(r'_1)=1$, where $r_0\in(0,\frac{\sqrt2}{2})$ is as in Lemma \ref{le2}(2). Then $f$ is strictly decreasing on $(0,r_1)$, $(\frac{\sqrt2}{2},r'_1)$ and strictly increasing on $(r_1,\frac{\sqrt2}{2})$ , $(r'_1,1)$. Thus,
$$f(r_1)=f(r'_1)<f({\sqrt2}/{2}).$$
Since $f(0^+)=f(1^-)=\infty$, there exists $r_2\in(0,r_1)\cup(r'_1,1)$  such that
$$f({\sqrt2}/{2})<f(r_2).$$
Therefore, neither (\ref{th11}) nor (\ref{th12}) holds for all $r\in(0,1)$.

This completes the proof of Theorem \ref{th1}.
\end{proof}

\begin{lemma}\label{t1l1}
Let $r\in(0,1)$.

(1) The function $f(r)\equiv\frac{r'^4\arth\,r-r(1+r^2)}{r'^2\left((1+r^2)\arth\, r-r\right)}$ is strictly decreasing with range $(-\infty, -2)$.

(2) For $p \in \mathbb{R}$ define
$$h_p(r)\equiv1+p r'^2\frac{\arth\, r}{r}-(1+r^2)\frac{\arth\, r}{r}.$$

(i) If $p\geq-2$, then the range of $h_p$ is $(-\infty,p)$.

(ii) If $p<-2$, then the range of $h_p$ is $(-\infty, C(p)]$,
where $C(p)\equiv\sup\limits_{0<r<1}{h_p(r)}\in (p\,,-1)$ with $\lim\limits_{p\to-2}C(p)=-2$ and $\lim\limits_{p\to-\infty}C(p)=-\infty$.
\end{lemma}

\begin{proof}
(1) Let $f_1(r)= r'^2\arth r-\frac{r(1+r^2)}{r'^2}$ and $f_2(r)=(1+r^2)\arth r-r$, then $f_1(0^+)=f_2(0^+)=0$.
By differentiation, we have
$$
\frac{f'_1(r)}{f'_2(r)}=-1-\frac{2r}{r'^2(r+r'^2\arth\,r)}=-1-\frac{2}{r'^2\left(1+\frac{f_3(r)}{f_4(r)}\right)},
$$
where $f_3(r)=\arth\,r$ and $f_4(r)=\frac{r}{r'^2}$. It is easy to see that $f_3(0^+)=f_4(0^+)=0$, then
$$
\frac{f'_3(r)}{f'_4(r)}=\frac{r'^2}{1+r^2},
$$
which is strictly decreasing. Hence by Lemma \ref{lhr}, $\frac{f'_1(r)}{f'_2(r)}$ is strictly decreasing and so is $f$ with $f(0^+)=-2$. Since $(1+r^2)\arth r-r>0$ and by l'H$\rm \hat{o}$pital's Rule,
\begin{eqnarray}\label{r'arth}
\lim_{r\rightarrow1^-}\frac{\arth r}{r'^{-2}}=\lim_{r\rightarrow1^-}\frac{r'^2}{2r}=0,
\end{eqnarray}
we get $f(1^-)=-\infty$.

\medskip
(2) By Lemma \ref{le1}(1)  and (\ref{r'arth}), it is easy to see that $h_p(0^+)=p$ and $h_p(1^-)=-\infty$.

Next by differentiation, we have
$$
h'_p(r)=\frac 1r\left((1+r^2)\frac{\arth\, r}{r}-1\right)\left(f(r)-p\right),
$$
where $f(r)$ is as in (1).

If $p\geq-2$, and by (1), we see that $p>f(r)$, which implies that  $h_p$ is strictly decreasing and hence $h_p(r)<p$.

If $p<-2$, since the range of $f$ is $(-\infty,-2)$, we see that
there exists exactly one point $r_0\in(0,1)$ such that $p=f(r_0)$. Then $h_p$ is increasing on $(0, r_0)$ and decreasing on $(r_0, 1)$. Since
$$
h_p(r)=-1+2\left(1-\frac{\arth\,r}{r}\right)+(p+1)r'^2\frac{\arth\,r}{r}<-1,
$$
by the continuity of $h_p$, there is a continuous function
$$C(p)\equiv\sup\limits_{0<r<1}{h_p(r)}$$
with $p<C(p)<-1$,\,$\lim\limits_{p\to-2}C(p)=-2$ and $\lim\limits_{p\to-\infty}C(p)=-\infty$.
\end{proof}

\begin{lemma}\label{t1l2}
Let $p\,,q\in\mathbb{R}$, $r\in(0,1)$, and let $C(p)$ be as in Lemma \ref{t1l1}(2). Let
$$g_{p,q}(r)\equiv\frac{\arth^{q-1}r}{r^{p-1}r'^2}.$$
(1) If $p\geq-2$, then $g_{p,q}$ is strictly increasing for each $q\geq{p}$, and $g_{p,q}$ is not monotone for any $q<p$.\\
(2) If $p<-2$, then $g_{p,q}$ is strictly increasing for each $q\geq{C(p)}$, and $g_{p,q}$ is not monotone for any $q<C(p)$.
\end{lemma}
\begin{proof} By logarithmic differentiation in $r$,
$$\frac{g'_{p,q}(r)}{g_{p,q}(r)}=\frac{1}{r'^2\arth\, r}(q-h_p(r)),$$
where $h_p(r)$ is as in Lemma \ref{t1l1}(2). Hence the results immediately follow from Lemma \ref{t1l1}(2).
\end{proof}

\medskip

The following theorem studies the $H_{p,q}$-convexity of $\arth$.

\begin{theorem}\label{ath1}
The inverse hyperbolic tangent function $\arth$  is strictly $H_{p,q}$-convex on $(0,1)$ if and only if $(p,q)\in{D_1}\cup{D_2}$,
where $$D_1=\{(p,q)|-2\leq{p}<+\infty,\, p\leq q<+\infty\},$$
$$D_2=\{(p,q)|-\infty<p<-2,\,C(p)\leq q<+\infty\},$$
and $C(p)$ is a continuous function as in Lemma \ref{t1l1}(2).
There are no values of $p$ and $q$ for which $\arth$ is $H_{p,q}$-concave on the whole interval $(0,1)$.
\end{theorem}

\begin{proof}  The proof is divided into the following four cases.

{\it Case 1.} $p\neq0$ and $q\neq0$.
We may suppose that $0<x\leq y<1$. Define
$$F(x,y)=\arth^{q}\left(H_p(x,y)\right)-\frac{\arth^{q}x+\arth^{q}y}{2}.$$
Let $t=H_p(x,y)$, then $\frac{\partial t}{\partial x}=\frac 12(\frac xt)^{p-1}$. If $x<y$, we see that $t>x$.
By differentiation, we have
$$\frac{\partial F}{\partial x}=\frac{q}{2}x^{p-1}\left(\frac{\arth^{q-1}t}{t^{p-1}t'^2}-\frac{\arth^{q-1}x}{x^{p-1}x'^2}\right).$$

{\it Case 1.1.} $p\geq-2$, $q\geq{p}$, and $pq\neq0$.
By Lemma \ref{t1l2}(1), $\frac{\partial{F}}{\partial{x}}<0$ if $q<0$ and $\frac{\partial{F}}{\partial{x}}>0$ if $q>0$.
Then $F(x,y)$ is strictly decreasing and $F(x,y)\geq F(y,y)=0$ if $q<0$, and $F(x,y)$ is strictly increasing and $F(x,y)\leq F(y,y)=0$ if $q>0$.
Hence we have
$$\arth(H_p(x,y))\leq H_q(\arth\, x,\arth\, y)$$
with equality if and only if $x=y$.

In conclusion, $\arth$ is strictly $H_{p,q}$-convex on $(0,1)$ for $(p,q)\in\{(p,q)|-2{\leq}p<0, p\leq q<0\}\cup\{(p,q)|-2{\leq}p<0, q>0\}\cup\{(p,q)|0<p<+\infty,q\geq p\}$.

{\it Case 1.2.} $p\geq-2$, $q<p$, and $pq\neq0$.
By Lemma \ref{t1l2}(1), with an argument similar to Case 1.1, it is easy to see that
$\arth$ is neither $H_{p,q}$-concave nor $H_{p,q}$-convex on the whole interval $(0,1)$.

{\it Case 1.3.} $p<-2$, $q\geq{C(p)}$, and $pq\neq0$.
By Lemma \ref{t1l2}(2), $\frac{\partial{F}}{\partial{x}}<0$ if $q<0$ and $\frac{\partial{F}}{\partial{x}}>0$ if $q>0$.
Then $F(x,y)$ is strictly decreasing and $F(x,y)\geq F(y,y)=0$ if $q<0$, and $F(x,y)$ is strictly increasing and $F(x,y)\leq F(y,y)=0$ if $q>0$.
Hence we have
$$\arth(H_p(x,y))\leq H_q(\arth\, x,\arth\, y)$$
with equality if and only if $x=y$.

In conclusion, $\arth$ is strictly $H_{p,q}$-convex on $(0,1)$ for $(p,q)\in\{(p,q)|p<-2, C(p)\leq q<0\}\cup\{(p,q)|p<-2, q>0\}$.

{\it Case 1.4.} $p<-2$, $q<{C(p)}$, and $pq\neq0$.
By Lemma \ref{t1l2}(2), with an argument similar to Case 1.3, it is easy to see that
$\arth$ is neither $H_{p,q}$-concave nor $H_{p,q}$-convex on the whole interval $(0,1)$.

{\it Case 2.} $p\neq0$ and $q=0$.
For $0<x\leq y<1$, let
$$F(x,y)=\frac{\arth^2(H_p(x,y))}{\arth\,{x}\,\arth\,{y}},$$
 and $t=H_p(x,y)$. If $x<y$, we see that $t>x$.
By logarithmic differentiation, we obtain
$$\frac1{F}\frac{\partial F}{\partial x}=x^{p-1}\left(\frac{(\arth\,{t})^{-1}}{t^{p-1}t'^2}-\frac{(\arth\,{x})^{-1}}{x^{p-1}x'^2}\right).$$

{\it Case 2.1.} $-2\leq{p}<0$ and $q=0>p$. By Lemma \ref{t1l2}(1), we have $\frac{\partial F}{\partial x}>0$ and $F(x,y)\leq F(y,y)=1$. Hence we have
$$\arth(H_p(x,y))\leq\sqrt{\arth\,{x}\,\arth\,{y}}$$
with equality if and only if $x=y$.

In conclusion, $\arth$ is strictly $H_{p,q}$-convex on $(0,1)$ for $(p,q)\in\{(p,q)|-2\leq{p}<0, q=0\}$.

{\it Case 2.2.} $p>0$ and $q=0<p$.  By Lemma \ref{t1l2}(1), with an argument similar to Case 2.1, it is easy to see that
$\arth$ is neither $H_{p,q}$-concave nor $H_{p,q}$-convex on the whole interval $(0,1)$.

{\it Case 2.3.} $p<-2$. We have $q=0\geq C(p)$, and by Lemma \ref{t1l2}(2), with an argument similar to Case 2.1, it is easy to see that
$\arth$ is $H_{p,q}$-convex on $(0,1)$ for $(p,q)\in\{(p,q)|p<-2, q=0\}$.

{\it Case 3.} $p=0$ and $q\neq0$.
For $0<x\leq y<1$, let
$$F(x,y)=\arth^q(\sqrt{xy})-\frac{\arth^qx+\arth^qy}{2},$$
 and $t=\sqrt{xy}$. If $x<y$, we have that $t>x$. By differentiation, we obtain
$$\frac{\partial F}{\partial x}=\frac{q}{2x}\left(\frac{\arth^{q-1}t}{t^{-1}t'^2}-\frac{\arth^{q-1}x}{x^{-1}x'^2}\right).$$

{\it Case 3.1.} $q>p=0$. By Lemma \ref{t1l2}(1), we have $\frac{\partial F}{\partial x}>0$ and $F(x,y)\leq F(y,y)=0$. Hence we have
$$\arth(\sqrt{xy})\leq H_q(\arth\,{x},\,\arth\,{y})$$
with equality if and only if $x=y$.

In conclusion, $\arth$ is strictly $H_{p,q}$-convex on $(0,1)$ for $(p,q)\in\{(p,q)|p=0, q>0\}$.

{\it Case 3.2.} $q<p=0$. By Lemma \ref{t1l2}(1), with an argument similar to Case 3.1, it is easy to see that
$\arth$ is neither $H_{p,q}$-concave nor $H_{p,q}$-convex on the whole interval $(0,1)$.

{\it Case 4.} $p=q=0$. By Case 1.1, for all $x\,,y\in(0,1)$, we have
$$\arth(H_p(x,y))\leq H_p(\arth\,{x},\,\arth\,{y}),\quad\mbox{for}\quad{p\geq-2}\quad\mbox{and}\quad p\neq0.$$
By the continuity of $H_p$ in $p$ and $\arth$ in $x$, we have
$$\arth(H_0(x,y))\leq H_0(\arth\,{x},\arth\,{y}).$$
In conclusion, $\arth$ is strictly $H_{0,0}$-convex on $(0,1)$.

This completes the proof of Theorem \ref{ath1}.
\end{proof}

\medskip

Setting $p=1=q$ in Theorem \ref{ath1}, we easily obtain the convexity of $\arth$.

\begin{corollary}
The inverse hyperbolic tangent function $\arth$ is strictly convex on $(0,1)$.
\end{corollary}
\medskip
By (\ref{arth}), Theorem \ref{ath1} has a simple application to the hyperbolic metric.

\begin{corollary}
Let $z\in S^1(H_p(|x|,|y|))$. Then for all $x,y\in\BB\setminus \{0\}$ and
$p\geq-2$
$$\rho(0,z)\leq H_p(\rho(0,x),\rho(0,y))$$
with equality if and only if $|x|=|y|$.
\end{corollary}

\section{Some Propositions for Hyperbolic Metric}

Next we give some geometric propositions for the hyperbolic metric.

\begin{proposition}
Let $c,d$ be arbitrary two points on the unit circle such that $0\,,c\,,d$ are noncollinear. Let $b\in[c,d]\cap \BB$,
 $\{a\}=[0,b]\cap J^*[c,d]$, and let $s$ be the midpoint of
the Euclidean segment $[c,d]$. Then\\
(1) $a$ is the hyperbolic midpoint of the hyperbolic segment $J[0,b]$;\\
(2) $s$ is on the hyperbolic circle $S_{\rho}(a,\rho(0,a))=\{z|\rho(z,a)=\rho(0,a)\}$.
\end{proposition}

\begin{figure}[h]
\begin{minipage}[t]{0.45\linewidth}
\centering
\includegraphics[width=6cm]{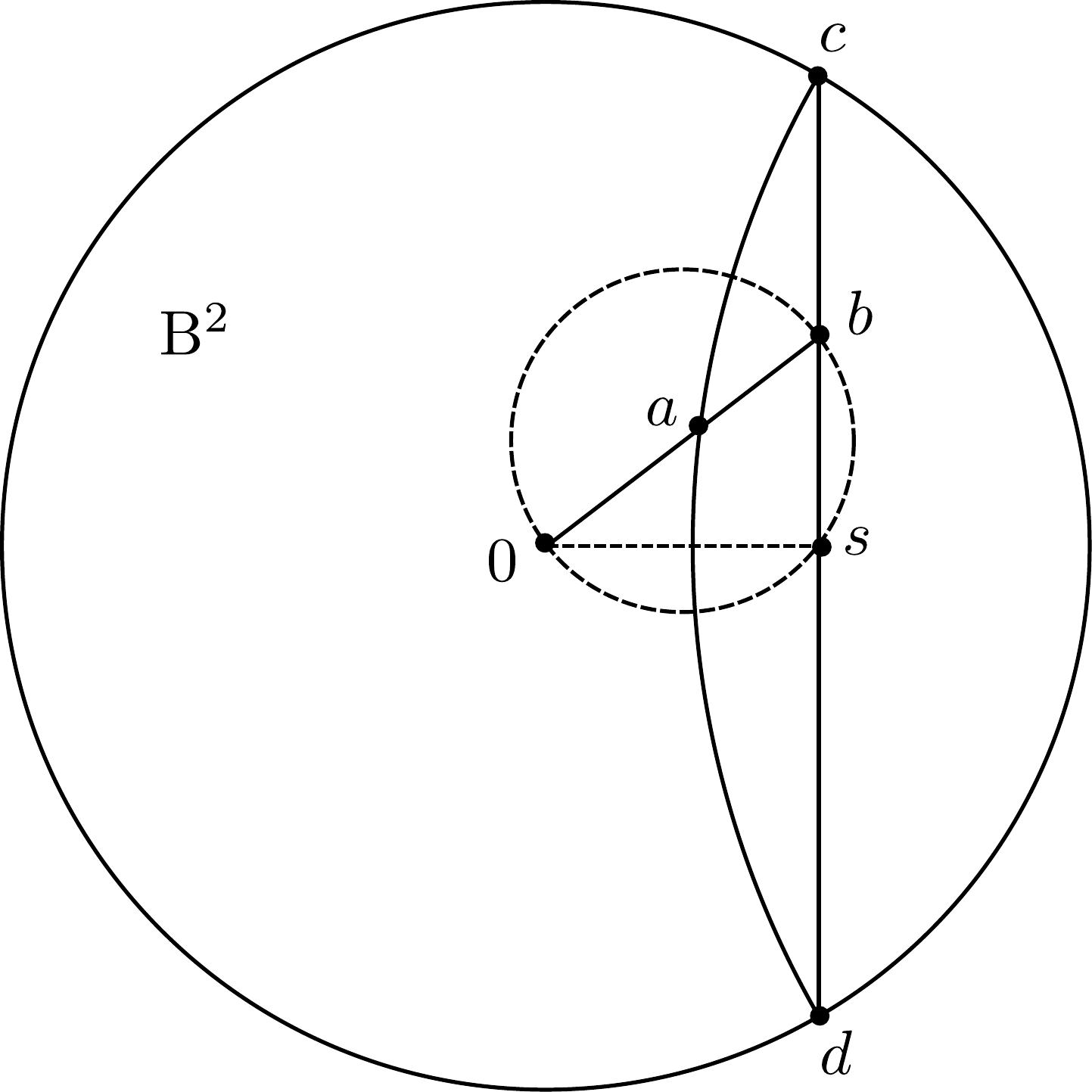}
\end{minipage}
\caption{The point $a$ is the hyperbolic midpoint of $J[0,b]\,.$ The circular arc $(cad)$ is orthogonal to the unit circle. Here $s$ is the Euclidean midpoint of the chord $[c,d]$. \label{abcd}}
\end{figure}

\begin{proof}
Without loss of generality, we may assume $c=e^{i\alpha}$ and $d=e^{-i\alpha}$, where $0<\alpha<\pi/2$.

(1) The point $a$ is the hyperbolic midpoint of the hyperbolic segment $J[0,b]$
\begin{eqnarray}\label{rhoab}
\Leftrightarrow \rho(0,b)=2\rho(0,a) \Leftrightarrow \log\frac{1+|b|}{1-|b|}=2\log\frac{1+|a|}{1-|a|} \Leftrightarrow |b|=\frac{2|a|}{1+|a|^2}.
\end{eqnarray}
Let $b=|b|e^{i\beta}$( $-\frac{\pi}{2}<\beta<\frac{\pi}{2}$).
Then by the orthogonality of $J^*[c,d]$ and the unit circle
$$
s=\cos\alpha=|b|\cos\beta=|b|\frac{|a|^2+|w|^2-r^2}{2|a| |w|},
$$
where $w=1/\cos\alpha$ and $r=\sqrt{|w|^2-1}$ are the center and the radius, resp., of the circle containing $c\,,d$ and orthogonal to $\partial\BB$.
Therefore, the last equality in (\ref{rhoab}) holds and $a$ is the hyperbolic midpoint
of $J[0,b]$.

\medskip
(2) It is easy to see that $[0,b]$ is also the Euclidean diameter of
$S_{\rho}(a,\rho(0,a))$ by geometric observation and (1). Therefore,
$s$ is on the hyperbolic circle $ S_{\rho}(a,\rho(0,a))$, see Figure \ref{abcd}.
\end{proof}


\begin{proposition}\label{a}
Let $J_1=J^*[e^{i\alpha},-e^{-i\alpha}]$,
$J_2=J^*[-e^{i\alpha},e^{-i\alpha}]$ be two hyperbolic geodesics in $\BB$,
$0<\alpha<\pi/2$. Let $\{t e_2\}=J_1\cap Y$, see Figure \ref{a1}. Then
\begin{eqnarray}\label{ad}
d_\rho(J_1,J_2)=\rho(-t e_2, t e_2)=2\log\frac{1+t}{1-t}.
\end{eqnarray}
\end{proposition}

\begin{proof}
Let $g: \BB\rightarrow \UH$ be a M\"obius  transformation which satisfies $g(i)=\infty$ and $g(-i)=0$. Then $g([-te_2, te_2])=[g(-te_2),g(te_2)]$ which is on the ray emanating from $0$ and perpendicular to $\partial \HH^2$.
By the orthogonality of $J_i$($i=1,2$) and $Y$, we get
$$J^*_1=g(J_1)=\Arc(g(-e^{-i\alpha})g(te_2)g(e^{i\alpha})) \,\ ,\,\
J^*_2=g(J_2)=\Arc(g(-e^{i\alpha})g(-te_2)g(e^{-i\alpha})).$$
Here $g(-e^{-i\alpha})=-g(e^{i\alpha})$ and $g(-e^{i\alpha})=-g(e^{-i\alpha})$, see Figure \ref{a2}.
For every rectifiable arc $\gamma_{xy}$, $x\in J^*_1$ and $y\in J^*_2$, by geometric observation we get
$$\int_{\gamma_{xy}}\frac{1}{d(z,\partial \UH)} |dz|\geq\int_{[g(-te_2),g(te_2)]}\frac{1}{d(z,\partial \UH)} |dz|,$$
and hence
\beq\label{adg}
d_\rho(J^*_1, J^*_2)=\rho(g(-t e_2),g(t e_2)).
\eeq
Since the hyperbolic distance is invariant under M\"obius transformations,
by (\ref{adg}) we get
\begin{eqnarray*}
d_\rho(J_1,J_2)=\rho(-t e_2, t e_2)=2\log\frac{1+t}{1-t}.
\end{eqnarray*}
\end{proof}
\begin{figure}[h]
\begin{minipage}[t]{0.45\linewidth}
\centering
\includegraphics[width=6cm]{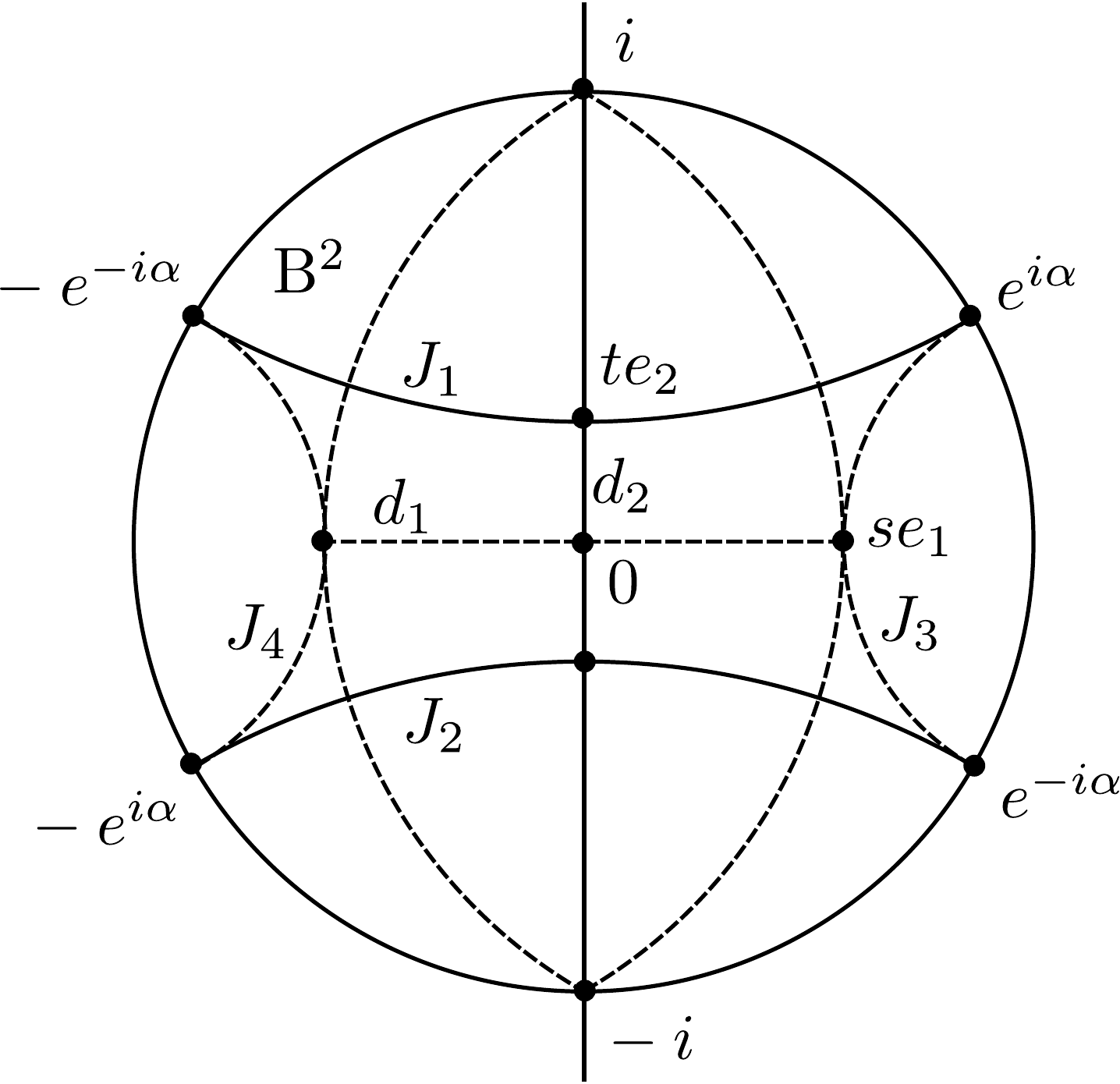}
\caption{\label{a1}}
\end{minipage}
\hfill
\hspace{1cm}
\begin{minipage}[t]{0.45\linewidth}
\centering
\includegraphics[width=7.8cm]{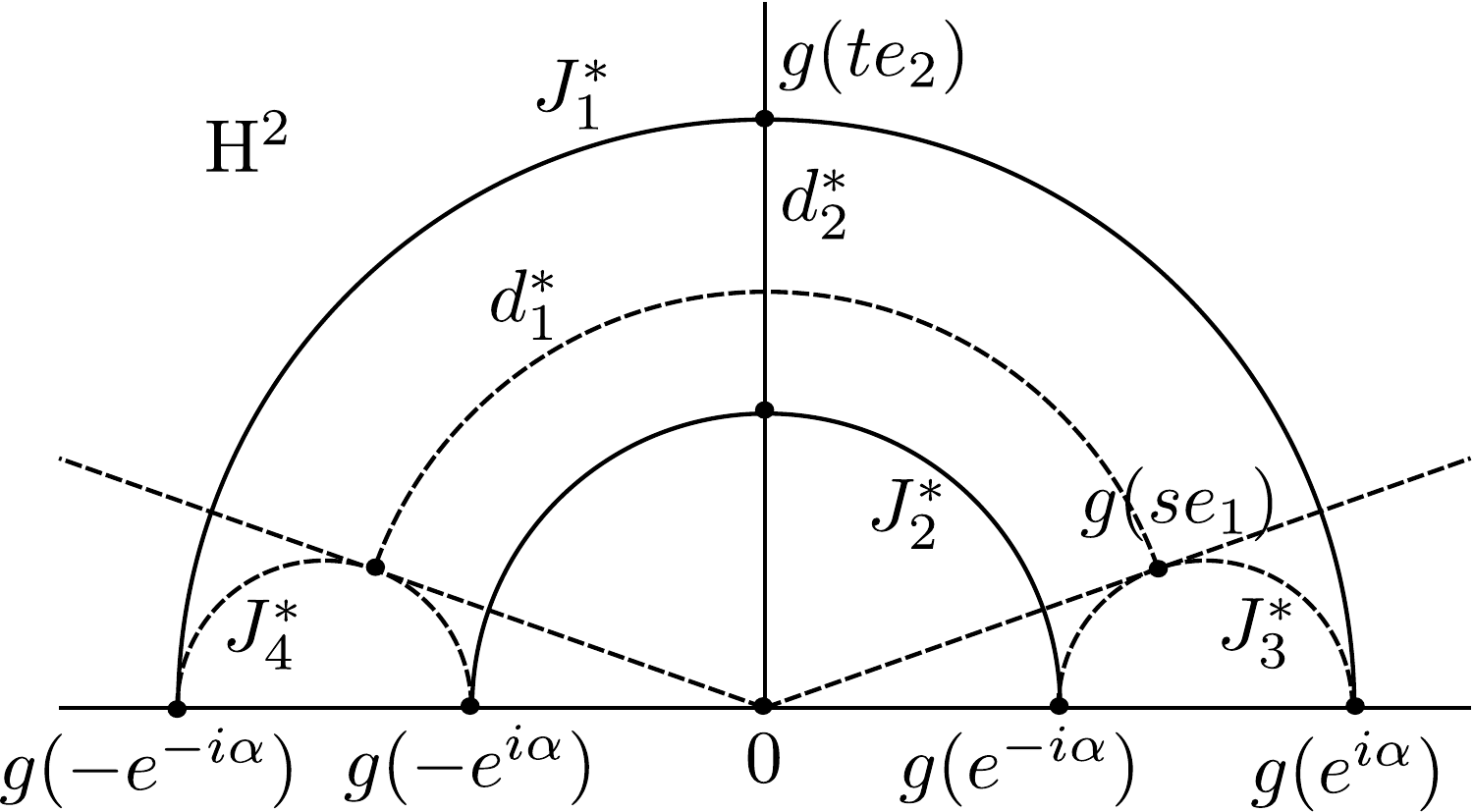}
\caption{\label{a2}}
\end{minipage}
\caption*{The M\"obius transformation $g$ maps the unit disk in Figure \ref{a1} onto the upper half plane in Figure \ref{a2}.}
\end{figure}

Let $J_3=J^*[e^{i\alpha},e^{-i\alpha}]$ and $J_4=J^*[-e^{-i\alpha},-e^{i\alpha}]$. Let $\{s e_1\}=J_3\cap X$. Since the hyperbolic distance is invariant under rotations, by Proposition \ref{a} we also obtain
\beq\label{j3j4}
d_\rho(J_3,J_4)=\rho(-se_1, se_1).
\eeq

By the proof of Proposition \ref{a}, $J^*_3=g(J_3)=\SArc(g(e^{i\alpha})g(e^{-i\alpha}))$ and $J^*_4=g(J_4)=\SArc(g(-e^{-i\alpha})g(-e^{i\alpha}))$.  Then
$g$ maps $\Arc((-i)(s e_1)i)$ to the ray emanating from $0$ and tangent to $J^*_3$ at the point $g(se_1)$.
Similarly, $g$ maps $\Arc((-i)(-s e_1)i)$ to the ray emanating from $0$ and tangent to $J^*_4$ at the point $g(-se_1)$.
Then we have
\beq\label{j3j4g}
d_\rho(J^*_3, J^*_4)=\rho(g(-se_1),g(se_1)).
\eeq

\begin{remark}
By (\ref{ad}) and (\ref{j3j4}), we find the two distances between the opposite sides for the hyperbolic quadrilaterals with four vertices $e^{i\alpha},\,-e^{-i\alpha},\,-e^{i\alpha},\,e^{-i\alpha}$ on the unit circle and the counterpart for the upper half plane by (\ref{adg}) and (\ref{j3j4g}).
\end{remark}

\section{Proof of Main Results }

\begin{proof}[Proof of Theorem \ref{Lamdd}]
Since the hyperbolic distance is M\"obius invariant, we may assume that $v_a=0$, $v_b$ is on $X$, $v_d$ is on $Y$ and $v_c=t e^{i\theta}$, $0<t\leq1$ and $0<\theta<\frac{\pi}{2}$ (see Figure \ref{Lamb}).
Then by (\ref{orar}) the circle $S^1(b,r_b)$ through $v_c$ and $\overline{v_c}$ is orthogonal to $\partial\BB$, where
$$
b=\frac{1+t^2}{2t\cos\theta}\,\,\,\,\,{\rm and}\,\,\,\,r_b=\frac{\sqrt{(1+t^2)^2-4t^2\cos^2\theta}}{2t\cos\theta}.
$$
By Proposition \ref{a}, we get
$$
d_1=\rho(0,v_b)=\log\frac{1+(b-r_b)}{1-(b-r_b)}=\arth\left(\frac{2t}{1+t^2}\cos\theta\right).
$$
Similarly, we get
$$
d_2=\rho(0,v_d)=\arth\left(\frac{2t}{1+t^2}\sin\theta\right).
$$
Then
$$d_1d_2=\arth(Lr)\arth(Lr'),$$
where $L=\frac{2t}{1+t^2}={\rm th}\rho(0,v_c)\in(0,1]$ and $r=\cos\theta\in(0,1)$. By Lemma \ref{lecr}(2), we have
$$d_1d_2\leq \left(\arth\left(\frac{\sqrt 2}{2}L\right)\right)^2.$$

This completes the proof of Theorem \ref{Lamdd}.
\end{proof}

\medskip

\begin{proof}[Proof of Theorem \ref{Lamdad}]
By the proof of Theorem \ref{Lamdd}, we have
$$d_1+d_2=\arth(Lr)+\arth(Lr'),$$
where $L\in(0,1]$ and $r\in(0,1)$.
Then by Lemma \ref{lecra}, the desired conclusion follows.
\end{proof}

\medskip

\begin{proof}[Proof of Corollary \ref{dd}]
There exists a M\"obius transformation $g$ which maps $a,b,c,d$ to
$e^{i\alpha}$, $-e^{-i\alpha}$, $-e^{i\alpha}$, $e^{-i\alpha}$, resp., where $\alpha=\arccos\sqrt{1/|a,b,c,d|}\in(0\,,\pi/2)$, see Figure \ref{a1}.
By Proposition \ref{a} and the proof of Theorem \ref{Lamdd}, we have
\beq\label{d2}
d_1=2 \arth\,r\,\,\,\,{\rm and}\,\,\,\,d_2=2 \arth\,r'
\eeq
where $r=\cos\alpha$.

Therefore, by Theorem \ref{Lamdd} and Theorem \ref{Lamdad}, we have
$$d_1d_2=4 (\arth\,r)(\arth\,r')\leq \left(2 \arth({\sqrt{2}}/{2})\right)^2$$
and
$$d_1+d_2=2 (\arth\,r+\arth\,r')\geq 4\arth({\sqrt{2}}/{2}).$$
The equalities hold if and only if $\alpha=\frac{\pi}{4}$, namely, $|a,b,c,d|=\frac{1}{\cos^2\alpha}=2$.

This completes the proof of Corollary \ref{dd}.
\end{proof}

Let $G\,,G'$ be domains in $\overline{\Rn}$ and let $f: G\rightarrow G'$ be a homeomorphism. Then $f$ is $K$-quasiconformal if
$$M(\Gamma)/K\le M(f\Gamma)\le K M(\Gamma)$$
for every curve family $\Gamma$ in $G$, where $M(\Gamma)$ is the modulus of $\Gamma$, see \cite[10.9]{v}.

For $r \in (0,1)$ and $K \ge 1$, we define the distortion function
 $$
  \varphi_{K}(r) = \mu^{-1}(\mu(r)/K)\,,
$$
where $\mu(r)$ is the modulus of the planar Gr\"otzsch ring, see \cite[Exercise 5.61]{v}.

\begin{lemma}\label{leqc}{\rm  \cite[Theorem 1.10] {bv} }
Let $f: \BB\rightarrow\BB$ be a $K$-quasiconformal mapping with $f \BB= \BB$, and let $\rho$ be the hyperbolic metric of $\BB$. Then
$$\rho(f(x),f(y))\leq A(K)\max\{\rho(x,y),\rho(x,y)^{1/K}\}$$
for all $x,y\in \BB$, where $A(K)=2 \arth (\varphi_K({\rm th} \frac 12))$ and
$$K\leq u(K-1)+1\leq\log({\rm ch} (K {\rm arch}(e)))\leq A(K)\leq v(K-1)+K$$
with  $u= {\rm arch}(e){\rm th} ( {\rm arch}(e))>1.5412$ and $v=\log(2(1+\sqrt{1-{1}/{e^2}}))<1.3507$. In particular,
$A(1)=1$.
\end{lemma}

\medskip

\begin{proof}[Proof of Theorem \ref{thqcLam}]
In the same way as in the proof of Theorem \ref{Lamdd}, we still assume that $v_a=0$, $v_b$ is on $X$, $v_d$ is on $Y$ and $v_c=t e^{i\theta}$, $0<t\leq1$ and $0<\theta<\frac{\pi}{2}$ (see Figure \ref{Lamb}). Then
$$d_1=\rho(0,v_b)=\arth(Lr)\,\,\,{\rm and} \,\,\,d_2=\rho(0,v_d)=\arth(Lr'),$$
where $0<L\leq1$ and $0<r<1$.

By Lemma \ref{leqc}, we have
\begin{eqnarray*}
D_1D_2&\leq&\rho(f(0),f(v_b))\rho(f(0),f(v_d))\\
&\leq& A(K)^2\max\{\rho(0,v_b),\rho(0,v_b)^{1/K}\}\cdot\max\{\rho(0,v_d),\rho(0,v_d)^{1/K}\}\\
&=& A(K)^2\max\{d_1,d_1^{1/K}\}\cdot\max\{d_2,d_2^{1/K}\}.
\end{eqnarray*}

(1) $0<L\leq \frac{e^2-1}{e^2+1}\approx0.761594$.
This implies that $d_1<1$ and $d_2<1$. Then by Theorem \ref{Lamdd},
$$D_1D_2\leq A(K)^2(d_1d_2)^{1/K}\leq A(K)^2 \left(\arth\left(\frac{\sqrt{2}}{2}L\right)\right)^{2/K}.$$

(2) $\frac{e^2-1}{e^2+1}<L\leq1$.

{\it Case 1.} $\frac{0.761594}{L}\approx\frac 1L\frac{e^2-1}{e^2+1}= r_L<r<1$.
This implies that $d_1>1$ and $d_2<1$. Then
$$
d_1d_2^{1/K}=\arth(L r)(\arth(L r'))^{1/K}\equiv F_{L,K}(r).
$$

By logarithmic differentiation, we have
\begin{eqnarray}\label{c22}
\frac{F'_{L,K}(r)}{F_{L,K}(r)}=\frac{L r}{r'(1-(L r')^2)\arth (L r')}\left(\frac{f_L(r)}{f_L(r')}-\frac 1K\right),
\end{eqnarray}
where $f_L(r)=\frac{1-(L r')^2}{r\,\arth(L r)}$. By Lemma \ref{lecr}(1), $\frac{f_L(r)}{f_L(r')}$ is strictly decreasing from
$(r_L, 1)$ onto $(0, \frac {1}{M_L})$. Here $M_L=\frac{f_L(r'_L)}{f_L(r_L)}>1$ since $r_L\approx\frac{0.761594}{L}\geq 0.761594>0.707107\approx\frac{\sqrt 2}{2}$.

{\it Case 1.1} $1\leq K\leq M_L$.
By (\ref{c22}), we have $F'_{L,K}(r)\leq 0$ and hence $F_{L,K}(r)$ is strictly decreasing on $(r_L,1)$.
Therefore,
\beq\label{qc11}
d_1d_2^{1/K}\leq F_{L,K}(r_L)=\arth(L r_L)(\arth(L r'_L))^{1/K}.
\eeq

{\it Case 1.2} $K>M_L$.
There exists exactly one point $r_L(K)\in(r_L, 1)$ such that $\frac{f_L(r_L(K))}{f_L(\sqrt{1-r_L(K)^2})}=\frac 1K$. Then $F_{L,K}$ is strictly increasing on $(r_L, r_L(K))$ and strictly decreasing on $(r_L(K),1)$. Therefore,
\begin{eqnarray}\label{qc22}
d_1d_2^{1/K}\leq F_{L,K}(r_L(K))=\arth\left(L r_L(K)\right)\left(\arth\left(L\sqrt{1- r_L(K)^2}\right)\right)^{1/K}.
\end{eqnarray}

{\it Case 2.} $\sqrt{1-r^2_L}<r<r_L$.
This implies that  $d_1<1$ and $d_2<1$. Since $\frac{\sqrt 2}{2}\approx0.707107\in (\sqrt{1-r^2_L},r_L)$ , by Theorem \ref{Lamdd} we have
\beq\label{qc33}
(d_1d_2)^{1/K}\leq \left(\arth\left(\frac{\sqrt{2}}{2}L\right)\right)^{2/K}.
\eeq

{\it Case 3.} $0<r<\sqrt{1-r^2_L}$.
This implies that  $d_1<1$ and $d_2>1$.
Putting $p=r'$, we have $r_L<p<1$ and
\begin{eqnarray*}
d_1^{1/K} d_2=\arth(L p)(\arth(L p'))^{1/K}.
\end{eqnarray*}
Hence Case 3 is the same as Case 1.

Therefore, by (\ref{qc11}) and (\ref{qc33}), we have
if $1\leq K\leq M_L$,
$$D_1D_2\leq A(K)^2\max\left\{
\arth(L r_L)\left(\arth\left(L \sqrt{1-r^2_L}\right)\right)^{1/K}\,,\left(\arth\left(\frac{\sqrt{2}}{2}L\right)\right)^{2/K}\right\}.$$
And
by (\ref{qc22}) and (\ref{qc33}), we have
if $K> M_L$,
$$D_1D_2\leq A(K)^2\max\left\{
\arth(L r_L(K))\left(\arth\left(L\sqrt{1- r_L(K)^2}\right)\right)^{1/K}\,, \left(\arth\left(\frac{\sqrt{2}}{2}L\right)\right)^{2/K}\right\}.$$

This completes the proof of Theorem \ref{thqcLam}.
\end{proof}

\medskip

\begin{proof}[Proof of Corollary \ref{thqc}]
First,  let $g$ be the same as in the proof of Corollary \ref{dd}. Let $\{s e_1\}=g(J^*[a,d])\cap X$, $\{t e_2\}=g(J^*[a,b])\cap Y$ and denote
$$z_1=g^{-1}(se_1),\,\,z_2=g^{-1}(te_2),\,\,z_3=g^{-1}(-se_1),\,\,z_4=g^{-1}(-te_2).$$

Since the hyperbolic distance is  M\"obius invariant, by the proof of Corollary \ref{dd}, we get
$$d_1=\rho(z_1,z_3)=2 \arth\, r\,\,\,{\rm and} \,\,\,d_2=\rho(z_2,z_4)=2 \arth\, r',\,\,\,0<r<1.$$

Then by Lemma \ref{leqc}, we have
\begin{eqnarray*}
D_1D_2&\leq&\rho(f(z_1),f(z_3))\rho(f(z_2),f(z_4))\\
&\leq& A(K)^2\max\{\rho(z_1,z_3),\rho(z_1,z_3)^{1/K}\}\cdot\max\{\rho(z_2,z_4),\rho(z_2,z_4)^{1/K}\}\\
&=& A(K)^2\max\{d_1,d_1^{1/K}\}\cdot\max\{d_2,d_2^{1/K}\}.
\end{eqnarray*}

{\it Case 1.} $0.886819\approx \frac{2 \sqrt e}{e+1}= r_1<r<1$.
This implies that  $d_1>1$ and $d_2<1$. Then
\begin{eqnarray*}
d_1 d_2^{1/K}=2^{1+1/K} \arth\, r (\arth\, r')^{1/K}\equiv2^{1+1/K}F_{1,K}(r).
\end{eqnarray*}
Let $M_1=\frac{f_1(r'_1)}{f_1(r_1)}$, where $f_1(r)=\frac{r}{\arth r}$ . By the proof of Case 1 in Theorem \ref{thqcLam}, we have the following conclusions.

{\it Case 1.1} $1\leq K\leq M_1$.
\beq\label{qc1}
d_1d_2^{1/K}\leq 2^{1+1/K} F_{1,K}(r_1)=2^{1+1/K} \arth\,r_1(\arth\,r'_1)^{1/K}.
\eeq

{\it Case 1.2} $K>M_1$.
There exists exactly one point $r_1(K)\in(r_1, 1)$ such that $\frac{f_1(r_1(K))}{f_1(\sqrt{1-r_1(K)^2})}=\frac 1K$. Therefore,
\beq\label{qc2}
d_1 d_2^{1/K}\leq 2^{1+1/K} F_{1,K}(r_1(K))= 2^{1+1/K}\arth\left(r_1(K)\right) \left(\arth\left(\sqrt{1-r_1(K)^2}\right)\right)^{1/K}.
\eeq

{\it Case 2.} $0.462117\approx\frac{e-1}{e+1}=\sqrt{1-r^2_1}<r<r_1\approx0.886819$.
This implies that  $d_1>1$ and $d_2>1$.
Since $\frac{\sqrt 2}{2}\approx0.707107\in (\sqrt{1-r^2_1},r_1)$ , by Corollary \ref{dd} we have
\begin{eqnarray}\label{qc3}
d_1d_2\leq\left(2 \log(\sqrt{2}+1)\right)^2.
\end{eqnarray}

{\it Case 3.} $0<r<\sqrt{1-r^2_1}\approx 0.462117$.
This implies that  $d_1<1$ and $d_2>1$. Putting $p=r'$, we have $r_1<p<1$ and
\begin{eqnarray*}
d_1^{1/K} d_2=2^{1+1/K} \arth\,p (\arth\, p')^{1/K}.
\end{eqnarray*}
Hence Case 3 is the same as Case 2.

Therefore, by (\ref{qc1}) and (\ref{qc3}), we have
if $1\leq K\leq M_1$,
$$D_1D_2\leq  A(K)^2\max\left\{2^{1+1/K} \arth\,r_1(\arth\,r'_1)^{1/K}, \left(2\log(\sqrt{2}+1)\right)^2\right\}.$$
And
by (\ref{qc2}) and (\ref{qc3}), we have
if $K> M_1$,
$$D_1D_2\leq A(K)^2\max\left\{2^{1+1/K}\arth\left(r_1(K)\right) \left(\arth\left(\sqrt{1-r_1(K)^2}\right)\right)^{1/K}, \left(2\log(\sqrt{2}+1)\right)^2\right\}.$$

This completes the proof of Corollary \ref{thqc}.
\end{proof}

\medskip

\subsection*{Acknowledgments}
The research of Matti Vuorinen was supported by the Academy of Finland,
Project 2600066611. The research of Gendi Wang was supported by CIMO
of Finland, Grant TM-10-7364.

\end{document}